\documentclass[12pt]{amsart}
\usepackage{amscd}
\def\frk{\frak}               

\def\pp{{\frk p}}

\def\qq{{\frk q}}

\def\mm{{\frk m}}

\def\Phi{{\frk n}}
\def\Phi{{\frk N}}
\def\opn#1#2{\def#1{\operatorname{#2}}} 
\opn\chara{char} \opn\length{\ell} \opn\pd{pd} \opn\rk{\lk}\opn\link{link}
\opn\projdim{proj\,dim} \opn\injdim{inj\,dim} \opn\rank{rank}
\opn\depth{depth} \opn\and{and} \opn\grade{grade}
\opn\height{height} \opn\embdim{emb\,dim} \opn\codim{codal}

\opn\Tr{Tr} \opn\bigrank{big\,rank}
\opn\superheight{superheight}\opn\lcm{lcm}
\opn\trdeg{tr\,deg}%
\opn\reg{reg} \opn\lreg{lreg} \opn\ini{in}
\opn\div{div} \opn\Div{Div} \opn\cl{cl} \opn\Cl{Cl}
\opn\Spec{Spec} \opn\Supp{Supp} \opn\supp{supp} \opn\Sing{Sing} \opn\Var{Var}
\opn\Ass{Ass} \opn\Min{Min}
\opn\Ann{Ann} \opn\Rad{Rad} \opn\Soc{Soc}
\opn\Im{Im} \opn \Bbbk{Bbbk} \opn\mdepth {mdepth}
 \opn\Ker{Ker} \opn\Coker{Coker} \opn\Am{Am} \opn \inf{inf}
\opn\Hom{Hom} \opn\Tor{Tor} \opn\Ext{Ext} \opn\End{End} \opn\cd{cd}
\opn\Aut{Aut} \opn\id{id} \opn\Att{Att} \opn\Assd{Assd}

\opn\nat{nat}
\opn\pff{pf}
\opn\Pf{Pf} \opn\GL{GL} \opn\SL{SL} \opn\mod{mod} \opn\ord{ord}
\opn\cl{cl} \opn\conv{conv} \opn\ext{ext} \opn\rad{rad}
\opn\star{star} \opn\red{red} \opn\H{H} \opn\bight{bight} \opn\Assh{Assh} \opn\Psupp{Psupp}
\opn\aff{aff} \opn\con{conv} \opn\relint{relint} \opn\st{st}
\opn\lk{lk} \opn\cn{cn} \opn\core{core} \opn\vol{vol}
\opn\link{link} \opn\star{star}
\opn\gr{gr}

\def\pot#1#2{#1[\kern-0.28ex[#2]\kern-0.28ex]}

\opn\dirlim{\underrightarrow{\lim}}
\opn\inivlim{\underleftarrow{\lim}}
\let\tensor=\otimes
\let\iso=\cong

\let\Dirsum=\bigoplus

\let\to=\rightarrow

\def\Implies{\ifmmode\Longrightarrow \else
     \unskip${}\Longrightarrow{}$\ignorespaces\fi}
\def\implies{\ifmmode\Rightarrow \else
     \unskip${}\Rightarrow{}$\ignorespaces\fi}
\def\iff{\ifmmode\Longleftrightarrow \else
     \unskip${}\Longleftrightarrow{}$\ignorespaces\fi}

\let\:=\colon
\newtheorem{Theorem}{Theorem}[section]

\newtheorem{Corollary}[Theorem]{Corollary}
\newtheorem{Proposition}[Theorem]{Proposition}
\newtheorem{Remark}[Theorem]{Remark}

\newtheorem{Example}[Theorem]{Example}

\newtheorem{Definition}[Theorem]{Definition}

\newtheorem{Fact}[Theorem]{Fact}
\let\epsilon\varepsilon
\let\phi=\varphi
\let\kappa=\varkappa
\def\qed{\ifhmode\textqed\fi
   \ifmmode\ifinner\quad\qedsymbol\else\dispqed\fi\fi}
\def\textqed{\unskip\nobreak\penalty50
    \hskip2em\hbox{}\nobreak\hfil\qedsymbol
    \parfillskip=0pt \finalhyphendemerits=0}
\def\dispqed{\rlap{\qquad\qedsymbol}}

\opn\dis{dis}
\def\pnt{{\raise0.5mm\hbox{\large\bf.}}}

\begin{document}
\title{ Maximal depth property of finitely generated modules}

\author{ Ahad Rahimi}

.

\address{ Ahad Rahimi, Department of Mathematics, Razi University, Kermanshah, Iran}\email{ahad.rahimi@razi.ac.ir}

\begin{abstract}
Let $(R,\mm)$ be a Noetherian local ring and $M$ a finitely generated $R$-module.  We say $M$ has maximal depth if there is an associated prime $\pp$ of $M$  such that $\depth M=\dim R/\pp$. In this paper,  we study finitely generated modules with maximal depth. It is shown that the maximal depth property is preserved under some important module operations. Generalized Cohen--Macaulay modules with maximal depth are classified. Finally, the attached primes of $H^{i}_{\mm}(M)$ are considered for $i<\dim M$.

\end{abstract}

\maketitle

\section*{Introduction}
Let $K$ be a field and $(R,\mm)$ be a Noetherian local ring, or a standard graded $K$-algebra with graded
maximal ideal $\mm$. Let $M$ be a finitely generated $R$-module.
A basic fact in Commutative algebra says that
\[
\depth M\leq \min\{\dim R/\pp: \pp \in \Ass(M)\}.
\]
We set $\mdepth_R M=\min\{\dim R/\pp: \pp \in \Ass(M)\}$. For simplicity, we write $\mdepth M$ instead of $\mdepth_RM$. We say $M$ has {\em maximal depth} if the equality holds, i.e., $\depth M =\mdepth M.$ In other words, there is an associated prime $\pp$ of $M$  such that $\depth M=\dim R/\pp$. Cohen--Macaulay modules as well as {\em sequentially Cohen--Macaulay} modules are examples of modules with maximal depth. Of course this class is rather large.  In Example \ref{C8} the ring $R$ is not sequentially Cohen--Macaulay and has maximal depth. This concept has already been working with several authors and some known results in this regards are as follows: If $I \subset S$ is a generic monomial ideal, then it has maximal depth, see \cite[Theorem 2.2]{MSY}. If a monomial ideal $I$ has maximal depth, then so does its polarization, see \cite{FT}.

In this paper, we study finitely generated modules to have maximal depth. Some classifications of this kind of modules are given.

 In the preliminary section, we give some facts about $\mdepth M$. For the {\em dimension filtration} $\mathcal{F}$: $0=M_0\subset M_1 \subset \dots  \subset M_d=M$ of $M$, we observe that all the $M_i$ have the same $\mdepth$, namely this number is the smallest $j$ for which $M_j\neq 0$. From this fact, we deduce that if $M$ is sequentially Cohen--Macaulay,  then $M$ has maximal depth, see Proposition \ref{mdepth}.

 In Section 2 , we show that the maximal depth property is preserved under tensor product, regular sequence and direct sum.
Let $A$ and $B$ be two Noetherian algebras over an algebraically closed field $K$. Let $M$
and $N$ be nonzero finitely generated modules over $A$ and $B$, respectively. We set  $R := A\tensor_K B$.  Then by using a result of \cite {HNTT} we show $M\tensor_KN$ as $R$-module has maximal depth  if and only if  $M$ and $N$ have maximal depth.

Next, it is shown that if $M$ has maximal depth and  $x=x_1, \dots, x_n$ is an $M$-sequence in $R$, then so does $M/xM$. The converse is not true in general, see Example~\ref{regularexample}.

 In the following section, we observe that if $M$ has maximal depth with $\depth M>0$, then $H^{\depth M}_{\mm}(M)$ is not finitely generated, see Proposition \ref{depthnotfg}. This is used to classify all {\em generalized Cohen--Macaulay}
  modules with maximal depth. In fact, we prove the following: Let $M$ be a generalized Cohen--Macaulay module with $\depth M>0$. Then $M$ has maximal depth is equivalent to say that "$M$ is sequentially Cohen--Macaulay" and this is the same as "$M$ is Cohen--Macaulay".

If $M$ is sequentially Cohen--Macaulay and $H^i_{\mm}(M)\neq 0$ for some $i$, then $H^i_{\mm}(M)$ is not finitely generated.
 Inspired by this fact and Proposition \ref{depthnotfg},  we may ask the following question in Remark \ref{not}: Assume $M$ has maximal depth and $H^i_{\mm}(M)\neq 0$. Does it follow $H^i_{\mm}(M)$ is not finitely generated?

We set $\Assd(M)=\{\pp \in \Ass(M): \depth M=\dim R/{\pp}\}$; so that $M$ has maximal depth if and only if $\Assd(M)\neq\emptyset$. In the final section, under the condition that $\pp \in \Supp(M)$ contains an element of $\Assd(M)$, we obtain several results. First of all, $M_{\pp}$ has maximal depth and $\depth M = \depth M_{\pp}+ \dim R/{\pp}$, see Proposition~\ref{localize}.
Secondly, we show
\[
\min \Att(H^{\depth M}_{\mm}(M))=\Assd(M).
\]
As a consequence, we explicitly obtain the attached primes $H^i_{\mm}(M)$ for all $i$ if $M$ is sequentially Cohen--Macaulay.
Note that the attached primes of top local cohomology are known. In fact, $\Att(H^{\dim M}_{\mm}(M))=\Assh(M)$ where $\Assh(M)=\{\pp \in \Ass(M): \dim M=\dim R/{\pp}\}$ and not so much is known about  $\Att(H^i_{\mm}(M))$ whenever $i<\dim M$.

At the end of this section we consider the following fact in \cite{TN}.
  For a finitely generated $R$-module $M$, one has $\Ann_R (M/\pp M)=\pp$  for all $\pp \in \Var(\Ann_R M)$. The dual property of this fact for an Artinian module $L$  is :
 \begin{equation*}
 \label{Psupp}
\hspace{3cm} \Ann (0 :_L \pp) =\pp  \quad \text{for all}\quad   \pp \in  \Var(\Ann L).\hspace{3cm}(*)
\end{equation*}

In \cite{TN} it is shown that the local cohomology module $H^i_{\mm}(M)$  satisfies the property $(*)$ with some conditions on $R$. Under our condition that $\pp \in \Supp(M)$ contains an element of $\Assd(M)$, we show $H^{\depth M }_{\mm}(M)$  satisfies $(*)$. As a consequence, if $M$ is sequentially Cohen--Macaulay, then $H^i_{\mm}(M)$  satisfies $(*)$ for all $i$.
\section{Preliminaries}

Let $K$ be a field and $(R,\mm)$ be a Noetherian local ring, or a standard graded $K$-algebra with graded
maximal ideal $\mm$. Let $M$ be a finitely generated $R$-module.
It is a classical fact that
\[
\depth M\leq \min\{\dim R/\pp: \pp \in \Ass(M)\},
\]
see \cite{BH}. We set $\mdepth_R M=\min\{\dim R/\pp: \pp \in \Ass(M)\}$. For simplicity, we write $\mdepth M$ instead of $\mdepth_RM$.
Thus $\depth M\leq \mdepth M\leq \dim M$. Observe that $\depth(M)=0$ if and only if $\mdepth(M)=0$. Thus, if $\mdepth M=1$, then $\depth M=1$.
\begin{Definition}{\em
We say $M$ has {\em maximal depth} if the equality holds, i.e.,
\[
\depth M =\mdepth M.
\]
In other words, there is an associated prime $\pp$ of $M$  such that $\depth M=\dim R/\pp$.}
\end{Definition}

  Cohen--Macaulay modules are obvious example of modules of maximal depth. If $\depth M=0$, then $M$ has maximal depth. In  particular, any non-zero Artinian module $N$ has maximal depth. Indeed, let $\pp \in \Ass(N)$. The exact sequence $0\to R/\pp\to N$ of $R$-modules implies that $R/\pp$ is an Artinian domain. It follows that $R/\pp$ is a filed and hence $\mm \in \Ass(N)$. Therefore $\depth N=0.$

  Notice that,  if $M$ is unmixed, then $M$ has maximal depth if and only if $M$ is Cohen--Macaulay.

This is a known fact that sequentially Cohen--Macaulay modules have maximal depth. In the following, we give some facts about $\mdepth M$ and hence deduce that sequentially Cohen--Macaulay modules have maximal depth.

Let $M$ be an $R$-module of dimension $d$.
The {\em dimension filtration}
  $\mathcal{F}$:
$0=M_0\subset M_1 \subset
 \dots  \subset M_d=M$
 of $M$ is defined by the property that $M_i$  is the largest submodule of $M$ with $\dim M_i\leq i$  for $i = 0, \dots, d$.
 In a dimension filtration each $M_i$  is uniquely determined as follows, see \cite{S}.
\begin{Fact}
\label{primary}
 Let $M$ be a finitely generated $R$-module and $\mathcal{F}$ be the dimension filtration of $M$. Then
\[
M_i = \bigcap_{\dim R/\pp_j>i}N_j
\]
for all $i=0, \dots, d$. Here $0 =\bigcap_{j=1}^nN_j$ denotes a reduced primary decomposition of $0$ in $M$.
\end{Fact}
For all i, we set $\Ass^i(M) = \{\pp \in \Ass(M): \dim R/\pp = i\}$. The following characterization
of a dimension filtration is given in \cite{HP}.

\begin{Fact}
\label{Ass}
Let  $\mathcal{F}$ be a filtration of
$M$. The following conditions are equivalent:
\begin{itemize}
\item[{(a)}] $\Ass(M_i/M_{i-1}) = \Ass^i(M)$ for all i;
\item[{(b)}] $\mathcal{F}$ is the dimension filtration of $M$.
\end{itemize}
\end{Fact}

A finite filtration $\mathcal{F}$:
$0=M_0\subset M_1 \subset
 \dots  \subset M_r=M$
 of $M$ by
submodules is a {\em Cohen--Macaulay filtration} if
each quotient $M_i/M_{i-1}$ is Cohen--Macaulay and $0 \leq \dim M_1/M_0<\dim M_2/M_1< \dots< \dim M_r/M_{r-1}$.
 If $M$ admits a Cohen--Macaulay filtration, then we say $M$ is
{\em sequentially Cohen--Macaulay}. Note that if $M$ is sequentially Cohen--Macaulay, then the Cohen--Macaulay filtration is just the dimension filtration.

\begin{Proposition}
\label{mdepth}
Let  $\mathcal{F}$: $0=M_0\subset M_1 \subset \dots  \subset M_d=M$ be the dimension filtration of $M$. We set $t=\min\{i: M_i\neq 0 \}$. Then
 $\mdepth M_i=t$ for $i=t, \dots, d$.  Moreover, if $M$ is sequentially Cohen--Macaulay, then $M$ has maximal depth.
\end{Proposition}
\begin{proof}
 We first show that $\mdepth M=t$. By Fact \ref{Ass},  $\Ass(M_t)=\Ass^t(M)$. Hence  $\dim M_t=t$; so there exists $\pp \in \Ass(M_t)$ such that $\dim R/\pp=t$. Suppose $\mdepth M=s$. As $\pp \in \Ass(M)$, it follows that $s\leq t$. If $s<t$, then $M_s=0$. Hence $\Ass^s(M)=\emptyset$ by Fact \ref{Ass},  which is a contradiction. Therefore, $\mdepth M=t$.
Now observe that
\[
t=\mdepth M\leq \mdepth M_{d-1}\leq \dots \leq \mdepth M_t\leq \dim M_t=t.
\]
Thus, $\mdepth M_i=t$ for $i=t, \dots, d$.

To show the second part, let $M$ be sequentially Cohen--Macaulay. Thus the dimension filtration of  $\mathcal{F}$ is the Cohen--Macaulay filtration. As $M_t$ is Cohen--Macaulay, it has maximal depth and so $\depth M_t=\mdepth M_t$. Since $M$ is sequentially Cohen--Macaulay, it follows that $\depth M_i=\depth M$ for all $i$, see for instance \cite[Fact 2.3]{R1}. Thus $\depth M=\depth M_t=\mdepth M_t=\mdepth M$, as desired.
\end{proof}

\begin{Remark}
\label{ass1}{\em In Proposition \ref{mdepth}, we observe that
\[
\mdepth M=\min\{ \dim R/\pp: \pp \in \Ass(M_t)\}=\mdepth M_t.
\]
In a dimension filtration  $\mathcal{F}$, one has $\Ass(M_t)=\Ass(M)-\Ass(M/M_t)$, see \cite{S}.
}
\end{Remark}
In the following, we give an example which is not sequentially Cohen--Macaulay and has maximal depth. Moreover, in the dimension filtration not necessarily all the $M_i$ have the same depth.

\begin{Example}
\label{C8}
{\em Let $S=K[x_1, \dots,  x_8]$ be the standard graded polynomial ring  with the unique graded maximal ideal $\mm=(x_1, \dots, x_8)$.
Let $C_8$ be the cycle graph of length $8$ and $I$ its edge ideal in $S$. Thus,
\[
I=(x_1x_2, x_2x_3, x_3x_4, x_4x_5, x_5x_6, x_6x_7, x_7x_8, x_8x_1).
\]
By using, CoCoA \cite{CO}, the ideal $I$ has the minimal primary decomposition $I=\bigcap_{i=1}^{10}\pp_i$ where
 $\pp_1=(x_1, x_3, x_5, x_7)$, $\pp_2=(x_2, x_4, x_6, x_8)$, $\pp_3=(x_2, x_3, x_5, x_7, x_8)$, $\pp_4=(x_2, x_3, x_5, x_6, x_8)$, $\pp_5=(x_1, x_3, x_5, x_6, x_8)$, $\pp_6=(x_1, x_3, x_4, x_6, x_8)$, \;$\pp_7=(x_1, x_2, x_4, x_5, x_7)$, $\pp_8=(x_1, x_2, x_4, x_6, x_7)$, $\pp_9=(x_1, x_3, x_4, x_6, x_7)$ and \;\; $\pp_{10}=(x_2, x_4, x_5, x_7, x_8).$
   We set $R=S/I$.
    Fact \ref{primary} provides the dimension filtration
\[
0=R_0\subset R_1\subset R_2\subset R_3  \subset R_4=R,
 \]
for $R$ where $R_3=(\pp_1\cap \pp_2)/I$ and $R_1=R_2=0$.   By Remark \ref{ass1},
\[
\Ass(R_3)=\Ass(R)-\Ass(R/R_3)=\{\pp_3, \dots, \pp_{10}\}.
\]
Hence,  $\mdepth R_3=\mdepth R=3$. By \cite[Corollary 7.6.30]{JS}  we have  $\depth R=3$(or using CoCoA \cite{CO}). Therefore, $R$ has maximal depth.   We claim that $\depth R_3=2$.
Consider the exact sequence $0\to R_3\to R\to R' \to 0$ where $R'=R/R_3 \iso S/(\pp_1\cap \pp_2)$.  One has $\depth R'=1$. Hence the Depth lemma yields $\depth R_3=2$. So in the dimension filtration $R_3$ and $R_4$ have different depth. The ring $R$ is not sequentially Cohen--Macaulay, because in the dimension filtration of $R$ the quotients $R_3/R_2$ and $R'=R/R_3$ are not Cohen--Macaulay. Indeed, $2=\depth R_3 \neq \dim R_3=3$, and $1=\depth R' \neq \dim R'=4$.  Note also that the only sequentially Cohen--Macaulay cycles are $C_3$ and $C_5$, see \cite{FV}.
}
\end{Example}

\begin{Remark}{\em It is natural to define the following notion:
A dimension filtration $\mathcal{F}$:
$0=M_0\subset M_1 \subset
 \dots  \subset M_r=M$
 of $M$ by
submodules is a {\em maximal depth filtration} if
each quotient $M_i/M_{i-1}$ has maximal depth and $0 \leq \dim M_1/M_0<\dim M_2/M_1< \dots< \dim M_r/M_{r-1}$.
 If $M$ admits a maximal depth filtration, then we say $M$ has {\em sequentially maximal depth}. We notice that $M$ has sequentially maximal depth is equivalent to say that $M$ is sequentially Cohen--Macaulay. In fact, by Fact \ref{Ass} we have $\Ass(M_i/M_{i-1})=\{ \pp \in \Ass(M): \dim R/\pp=i\}$. Thus $M_i/M_{i-1}$ is unmixed. Consequently, $M_i/M_{i-1}$ has maximal depth is equivalent to say that $M_i/M_{i-1}$ is Cohen--Macaulay. }
\end{Remark}

\section{Tensor product, regular sequence and direct sum}
In this section, we show that the maximal depth property is preserved under some important module
operations.

Let $A$ and $B$ be two Noetherian algebras over a field $K$ such that $R := A\tensor_K B$
is Noetherian. The following fact gives the associated primes of $R$-modules of the form $M\tensor_K N$, where $M$
and $N$ are nonzero finitely generated modules over $A$ and $B$, respectively.
\begin{Fact}	
\label{ass}{\em Let $K$ be an algebraically closed field.  Then
\[
\Ass_R(M\tensor_KN)=\{ \pp_1+\pp_2: \pp_1 \in \Ass_{A}(M) \quad\text{and} \quad \pp_2\in \Ass_{B}(N)\},
\]
see, \cite[Corollary 2.8]{HNTT}.}
\end{Fact}

\begin{Proposition}
\label{joint}
Continue with the notation and assumptions as above,  $M\tensor_KN$ has maximal depth  if and only if  $M$ and $N$ have maximal depth.
\end{Proposition}
\begin{proof}
We set $L=M\tensor_KN$ and suppose $L$ has maximal depth. Thus there exists $\pp \in \Ass(L)$ such that $\depth L=\dim R/\pp$. This is a known fact that $\depth L=\depth M+\depth N$.  Fact \ref{ass}  provides $\pp=\pp_1+\pp_2$ where  $\pp_1 \in \Ass_{A}(M)$ and $\pp_2\in \Ass_{B}(N)$. Consequently,
\[
\depth M+\depth N=\dim R/(\pp_1+\pp_2)=\dim A/\pp_1+\dim B/\pp_2.
\]
Since  $\depth M\leq \dim A/\pp_1$ and $\depth N\leq \dim B/\pp_2$, it follows that  $\depth M= \dim A/\pp_1$ and $\depth N= \dim B/\pp_2$, and hence $M$ and $N$ have maximal depth. The converse is proved in the same way with using again Fact \ref{ass}.
\end{proof}

Next, we have the following
\begin{Proposition}
\label{regular}
Let $M$ be an $R$-module which has maximal depth. Let $x\in R$ be an $M$-regular element. Then $M/xM$ has maximal depth.
\end{Proposition}
\begin{proof}
By our assumption, there exists $\pp\in \Ass(M)$ such that $\depth M=\dim R/\pp$. We may assume $\depth M>1$. For $\pp \in \Ass(M)$ we choose $z\in M$ such that $Rz$ is maximal among the cyclic submodules of $M$ annihilated by $\pp$.
As in the proof of \cite[Proposition 1.2.13]{BH} we see that $\pp$ consists of zero divisors of $M/xM$.
Thus $\pp \subseteq \qq$ for some $\qq \in \Ass(M/xM)$. Note that $x \in \qq$. Since $x$ is $M$-regular, it follows that $x \not \in \pp$. Hence $\pp \neq\qq$.  Observe that
\begin{eqnarray*}
\depth M-1 & = & \depth M/xM \\
          & \leq & \dim R/{\qq} \\
          &<&  \dim R/{\pp}\\
          & =&  \depth M.
\end{eqnarray*}
Consequently,
\[
\depth M/xM = \dim R/{\qq}  \quad\text{where}\quad \qq \in \Ass(M/xM).
\]
 Therefore, $M/xM$ has maximal depth,
as desired.
\end{proof}

\begin{Corollary}
Let $M$ be an $R$-module which has maximal depth. Let $x=x_1, \dots, x_n\in R$ be an $M$-sequence. Then $M/xM$ has maximal depth.
\end{Corollary}
The following example shows that the converse of Proposition \ref{regular} is not true in general.
\begin{Example}
\label{regularexample}{\em
Let $S=K[x_1,x_2,x_3, x_4]$ be the standard graded polynomial
ring
 with the unique graded maximal ideal $\mm=(x_1, x_2, x_3, x_4)$. We set  $R=S/I$ where $I=\pp_1\cap\pp_2$
 with $\pp_1=(x_1,x_2)$ and $\pp_2=(x_3, x_4)$. The element $x=x_1+x_3\in S$ is $R$-regular. Consider the following isomorphism
 \[
 R/xR\iso S/(I, x).
 \]
 Using Macaulay2 (\cite{GSE}) gives us the associated primes of $T=S/(I, x)$ that is
 \[
 \{ (x_1, x_3), (x_2, x_4, x_1+x_3)\}.
 \]
One has that $\depth T=\mdepth T=1$. Hence $R/xR$ has maximal depth. On the other hand, $\depth R=1$ and $\mdepth R=2$. Thus, $R$ has no maximal depth. }
\end{Example}

\begin{Proposition}
Let $M_1, \dots, M_n$ be finitely generated $R$-modules.  Then  $\Dirsum_{i=1}^n M_i$ has maximal depth if and only if there exists $j$ such that $\depth M_j\leq \depth M_k$ for all $k$ and $M_j$ has maximal depth.
\end{Proposition}
\begin{proof}
Suppose  $\Dirsum_{i=1}^n M_i$  has maximal depth. Thus there exists an associated prime $\pp$ of  $\Dirsum_{i=1}^n M_i$ such that $\depth \big( \Dirsum_{i=1}^n M_i\big)=\dim R/\pp$. Note that $\depth \big( \Dirsum_{i=1}^n M_i\big)=\depth M_s$ where $\depth M_s \leq  \depth M_k$ for all $k$. Since $\pp \in \Ass\big( \Dirsum_{i=1}^n M_i\big)= \bigcup_{i=1}^n\Ass(M_i)$, it follows that $\pp \in \Ass(M_s)$ where $\depth M_s \leq \depth M_k$ for all $k$. Indeed, otherwise  $\pp \in \Ass(M_l)$ for some $l$ and there exists $h$ such that  $\depth M_h<\depth M_l$. Hence $\depth M_s\leq \depth M_h<\depth M_l\leq \dim R/\pp$, a contradiction. Therefore, the conclusion follows. The other implication is obvious.
\end{proof}
\section{Generalized Cohen--Macaulay modules with maximal depth}

Let $(R, \mm)$ is a local ring and $M$ a finitely generated $R$-module. Recall that $M$ is called {\em generalized Cohen--Macaulay } module if the local cohomology module $H^i_{\mm}(M)$ is of finite length for all $i<\dim M$.
 Recall that a prime ideal $\pp$ of $R$ is called an {\em attached prime} of $M$  if there exists a
quotient $Q$ of $M$  such that $\pp =\Ann(Q)$. The set of attached primes of $M$ will be denoted by $\Att(M)$. If $M$ is Artinian, this definition is the same as MacDonald's definition. Note also that $H^i_{\mm}(M)$ is Artinian for all $i\geq 0$,  see \cite{BS}.
\begin{Fact}
\label{fl}{\em
Let $L$ be an Artinian $R$-module. Then $L\neq 0 $ if and only if $\Att(L) \neq \emptyset$,  and $L\neq 0$ is of finite length if and only if $\Att(L)=\{ \mm \}$, see \cite[Corollary 7.2.12]{BS}.}
\end{Fact}

\begin{Fact}
\label{att}{\em
Let $\pp \in \Ass(M)$ with $\dim R/\pp=j$. Then $H^j_{\mm}(M)\neq 0$ and $\pp \in \Att(H^j_{\mm}(M))$, see \cite[Exercise 11.3.9]{BS}.}
\end{Fact}
\begin{Proposition}
\label{depthnotfg}
Assume $M$ has maximal depth with $\depth M>0$. Then $H^{\depth M}_{\mm}(M)$ is not finitely generated.
\end{Proposition}
\begin{proof}
Our assumption says that there exists $\pp\in \Ass(M)$ such that $\depth M=\dim R/\pp$.
 Fact \ref{att} provides $\pp \in \Att(H^{\depth M}_{\mm}(M))$  and of course $H^{\depth M}_{\mm}(M)\neq 0$.  On the contrary, suppose $H^{\depth M}_{\mm}(M)$ is finitely generated. As $H^{\depth M}_{\mm}(M)$ is always Artinian, it has finite length
 and so $\Att(H^{\depth M}_{\mm}(M))=\{\mm\}$  by  Fact \ref{fl}. Hence, $\pp=\mm$. Therefore, $\depth M=0$, a contradiction.
\end{proof}
We have the following classification of generalized Cohen--Macaulay modules with maximal depth.
\begin{Corollary}
Suppose $M$ is generalized Cohen--Macaulay module with $\depth M>0$. Then the following statements are equivalent:
\begin{itemize}
\item[{(a)}]  $M$ has maximal depth,
\item[{(b)}]  $M$ is sequentially Cohen--Macaulay,
\item[{(c)}]  $M$ is Cohen--Macaulay.
\end{itemize}
\end{Corollary}
\begin{proof}
The implication $(c)\Rightarrow (b)$ is obvious. The implication  $(b)\Rightarrow (c)$ follows from Proposition \ref{mdepth}. For $(a)\Rightarrow (c)$, we need to show $\depth M=\dim M$. Suppose $\depth M<\dim M$. Since $M$ is generalized Cohen--Macaulay, $H^{\depth M}_{\mm}(M)$ is finitely generated. This contradicts with Proposition \ref{depthnotfg}. Therefore, $\depth M=\dim M$.
\end{proof}


\begin{Example}{\em
Let $S=K[x_1,x_2,x_3, x_4]$ be the standard graded polynomial
ring
 with the unique graded maximal ideal $\mm=(x_1, x_2, x_3, x_4)$. We set  $R=S/I$ where $I=\pp_1\cap\pp_2$
 with $\pp_1=(x_1,x_2)$ and $\pp_2=(x_3, x_4)$.
   The exact sequence
  $0\rightarrow R \rightarrow S/\pp_1 \oplus S/\pp_2 \rightarrow S/{\mm}\rightarrow 0$
  yields $H^0_{\mm}(R)=0$ and $H^1_{\mm}(R)\iso S/{\mm}$. Hence $H^1_{\mm}(R)$ is finitely generated. Therefore, $R$ is generalized Cohen--Macaulay. As $\depth R=1$ and $\mdepth R=2$, the ring $R$ has no maximal depth. }
\end{Example}
We end this section with the following remark
\begin{Remark}
\label{not}{\em
Assume $M$ has maximal depth with $\depth M>0$. Then $H^{\depth M}_{\mm}(M)$ is not finitely generated.
On the other hand, since $\dim M=\dim R/\pp$ for some $\pp \in \Ass(M)$, it follows from  Fact \ref{att}  that $H^{\dim M}_{\mm}(M)$ is not finitely generated, too. All the non-vanishing local cohomology modules of a sequentially Cohen--Macaulay module are not finitely generated, see the isomorphisms (\ref{notfg}). So inspired by these facts, we may ask the following question: Assume $M$ has maximal depth and $H^i_{\mm}(M)\neq 0$. Does it follow $H^i_{\mm}(M)$ is not finitely generated?}
\end{Remark}

\section{Localization and Attached primes }
In this section, $(R, \mm)$ is a local ring and $M$ a finitely generated $R$-module.
We recall the following known fact as {\em Depth Inequality}.
\begin{Fact}
\label{Depth Inequality}{\em
Let  $\pp \in \Supp(M)$. Then the next inequality holds
\[
\depth M\leq \depth M_{\pp}+ \dim R/{\pp}.
\]}
\end{Fact}


We set $\Assd(M)=\{\pp \in \Ass(M): \depth M=\dim R/{\pp}\}$. Observe that $M$ has maximal depth if and only if $\Assd(M)\neq\emptyset$. We have the following relation between localization and maximal depth property of $M$.
\begin{Proposition}
\label{localize}
The following statements hold:
\begin{itemize}
\item[{(a)}] If $\pp \in \Ass(M)$, then $M_{\pp}$ has maximal depth.
\item[(b)] If $\pp \in \Supp(M)$ contains an element of $\Assd(M)$, then $M_{\pp}$ has maximal depth. Moreover,
\[
\depth M = \depth M_{\pp}+ \dim R/{\pp}.
\]
\end{itemize}
\end{Proposition}
\begin{proof}
(a): Let $\pp \in \Ass(M)$. Then $\pp R_{\pp} \in \Ass_{R_{\pp}}(M_{\pp})$. It follows that $\depth M_{\pp}=0$, and hence $M_{\pp}$ has maximal depth.

 (b): By our assumption, there exists $\qq \in \Ass(M)$ such that $\depth M=\dim R/\qq$ with $\qq \subseteq \pp$. Thus ${\qq}R_{\pp}\in \Ass_{R_{\pp}}(M_{\pp})$. Hence
 \begin{eqnarray*}
 \depth M_{\pp} & \leq & \dim R_{\pp}/{\qq}R_{\pp} \\
                & =&  \height {\pp} R_{\pp}/{\qq}R_{\pp}\\
                & =&  \height {\pp}/{\qq}.
\end{eqnarray*}
Observe that
\begin{eqnarray*}
\dim R/\qq =\depth M & \leq &  \depth M_{\pp}+\dim R/\pp \\
                     & \leq &     \height {\pp}/{\qq}+\dim R/{\pp}\\
                     &\leq &  \dim R/{\qq}-\dim R/{\pp}+\dim R/{\pp}\\
                     & =&     \dim R/{\qq}.
\end{eqnarray*}
Fact \ref{Depth Inequality} explains the first step in this sequence and the remaining steps are standard.
Consequently,
\[
\depth M_{\pp}=\dim R_{\pp}/{\qq}R_{\pp} \quad\text{for some}\quad  {\qq}R_{\pp}\in \Ass_{R_{\pp}}(M_{\pp}).
 \]
 Therefore,  $M_{\pp}$ has maximal depth. Furthermore, the desired equality also follows.
\end{proof}
\begin{Corollary}
If $M$ is Cohen--Macaulay, then $M_{\pp}$ has maximal depth for all $\pp \in \Supp(M)$. Moreover,
\[
\depth M = \depth M_{\pp}+ \dim R/{\pp}.
\]
\end{Corollary}
\begin{proof}
As $\pp \in \Supp(M)$, there exists $\qq \in \Ass(M)$ such that $\qq\subseteq \pp$. Since $M$ is Cohen--Macaulay, it is unmixed and so $\depth M=\dim R/{\qq}$. Hence $\qq \in \Assd(M)$.  Thus the assumption of Proposition \ref{localize} is satisfied. Therefore, the conclusion follows.
\end{proof}

Recall that by $\Assh(M)$ we mean the set of all associated primes $\pp \in \Ass(M)$ such that $\dim R/\pp=\dim M$. The attached primes of top local cohomology module are known. In fact,
 \begin{equation}
 \label{assh}
\Att(H^{\dim M}_{\mm}(M))=\Assh(M),
\end{equation}
 see \cite[Theorem 7.3.2]{BS}.
  In the following we consider the attached primes of $H^i_{\mm}(M)$ whenever $i<\dim M$. From now on,  $(R, \mm)$ is a local ring which is a homomorphic image of a Gorenstein local ring.
\begin{Proposition}
\label{attdepth}
Assume $\pp \in \Supp(M)$ contains an element of $\Assd(M)$. Then
\[
\min \Att(H^{\depth M}_{\mm}(M))=\{ \pp \in \Ass(M): \dim R/\pp=\depth M \}.
\]
\end{Proposition}
\begin{proof}
 We set $\depth M=r$. If  $r=0$, then $\mm \in \Ass(M)$ and $H^0_{\mm}(M)$ is of finite length. Hence  $\Att(H^0_{\mm}(M))=\{\mm \}$ by Fact \ref{fl}. Therefore, the result holds in this case. Now let $r>0$ and  $\pp \in \Assd(M)$. Thus $\pp \in \Ass(M)$ and $\dim R/\pp=r$. Fact \ref{att} provides $\pp \in \Att(H^r_{\mm}(M))$. We need to show that $\pp$ is minimal. Suppose $\qq \subseteq \pp$ where $\qq \in \Att(H^r_{\mm}(M))$. By Shifted Localization Principle
\[
\qq R_{\pp}\in \Att_{R_{\pp}}\big(H_{\pp R_{\pp}}^{0}(M_{\pp})\big)=\{\pp R_{\pp} \},
\]
see, \cite[Theorem 11.3.2]{BS}.
It follows that $\qq R_{\pp}=\pp R_{\pp}$ and hence $\pp=\qq$.

Now let $\pp \in \min \Att(H^r_{\mm}(M))$. Thus $\Ann M\subseteq \Ann H^r_{\mm}(M)\subseteq \pp$ and so $\pp \in \Supp(M)$. Our assumption implies $\qq \subseteq \pp$ where $\qq \in \Ass(M)$ and $\depth M=\dim R/\qq=r$. Hence  $\qq \in \Att(H^r_{\mm}(M))$ by Fact \ref{att}. Since $\pp$ is minimal in $\Att(H^r_{\mm}(M))$, it follows that $\pp=\qq$ and hence $\pp \in \Assd(M)$.
\end{proof}
As a consequence of Proposition \ref{attdepth} and (\ref{assh}) we have the following
\begin{Corollary}
\label{cm}
If $M$ is Cohen--Macaulay of dimension $d$, then
\[
\Att(H^d_{\mm}(M))=\min \Att(H^d_{\mm}(M))=\Assd(M)=\Ass(M).
\]
\end{Corollary}
 In the following, we have the attached prime $H^i_{\mm}(M)$ for all $i$ if $M$ is sequentially Cohen--Macaulay.
\begin{Proposition}
\label{scm}
Let  $\mathcal{F}$:
$0=M_0\subset M_1 \subset
 \dots  \subset M_d=M$ be a Cohen--Macaulay filtration of
$M$. Then
\[
\Att(H^i_{\mm}(M))=\Ass(M_i/M_{i-1}) \quad \text{for all}\quad i.
\]
Moreover,
\[
\bigcup_{i=0}^d \Att(H^i_{\mm}(M))=\Ass(M),
\]
where $d=\dim M$.
\end{Proposition}
\begin{proof}
For all $i$, the exact sequence $0 \to M_{i-1} \to M_i \to M_i/M_{i-1} \to 0$ yields
\begin{eqnarray}
\label{notfg}
H^i_{\mm}(M)\iso H^i_{\mm}(M_i) \iso H^i_{\mm}(M_i/M_{i-1}),
\end{eqnarray}
see \cite[Lemma 4.5]{S}. Since $M_i/M_{i-1}$ is Cohen-Macaulay, the first desired equality follows from Corollary \ref{cm}. The second equality follows from Fact \ref{Ass}.

\end{proof}

At the end of this section we consider the following fact in \cite{TN}. For a finitely generated $R$-module $M$, one has $\Ann_R (M/\pp M)=\pp$  for all $\pp \in \Var(\Ann_R M)$. The dual property of this fact for an Artinian module $L$  is :
 \[
\hspace{3cm} \Ann (0 :_L \pp) =\pp  \quad \text{for all}\quad   \pp \in  \Var(\Ann L). \hspace{3cm}(*)
\]

In \cite{TN} it is shown that the local cohomology module $H^i_{\mm}(M)$  satisfying the property $(*)$ is equivalent to say that $\Var(\Ann(H^i_{\mm}(M))) = \Psupp^i_R M$ where  $\Psupp^i_R M=\{ \pp \in \Spec(R): H^{i-\dim(R/\pp)}_{\pp R_{\pp}}(M_{\pp})\neq 0\}$, see \cite[Theorem 3.1]{TN}. It is also shown  in \cite{TN} that $H^i_{\mm}(M)$  satisfies $(*)$ with some conditions on $R$.

\begin{Proposition}
\label{Psupp1}
Assume $\pp \in \Supp(M)$ contains an element of $\Assd(M)$.  Then $H^{\depth M }_{\mm}(M)$  satisfies $(*)$.
\end{Proposition}
\begin{proof}
 We set $\depth M=r$. By \cite[Theorem 3.1]{TN} we need to show $\Var(\Ann(H^r_{\mm}(M))) = \Psupp^r_R M$. Let $\pp \in  \Var(\Ann(H^r_{\mm}(M)))$. Thus $\Ann M \subseteq \Ann(H^r_{\mm}(M))\subseteq \pp$ and so $\pp \in \Supp(M).$ By our assumption, there exists $\qq \in \Ass(M)$ such that $\qq \subseteq \pp$ and $\dim R/\qq =r$. Fact \ref{att} provides $\qq \in \Att(H^r_{\mm}(M))$. By Shifted Localization Principle
\[
\qq R_{\pp}\in \Att_{R_{\pp}}\big(H_{\pp R_{\pp}}^{r-\dim R/\pp}(M_{\pp})\big).
\]
It follows that $H_{\pp R_{\pp}}^{r-\dim R/\pp}(M_{\pp})\neq 0$ and hence $\pp \in \Psupp^r_R M$. The other inclusion is proved in the same way.
\end{proof}
\begin{Corollary}
\label{cm1}
If $M$ is Cohen--Macaulay of dimension $d$, then $H^d_{\mm}(M)$  satisfies $(*)$.
\end{Corollary}

\begin{Corollary}
\label{scm1}
If $M$ is sequentially Cohen--Macaulay, then $H^i_{\mm}(M)$  satisfies $(*)$ for all $i$.
\end{Corollary}
\begin{proof}
Let  $\mathcal{F}$:
$0=M_0\subset M_1 \subset
 \dots  \subset M_d=M$ be a Cohen--Macaulay filtration of
$M$. Notice that for all $i$ we have  $H^i_{\mm}(M) \iso H^i_{\mm}(M_i/M_{i-1})$, see (\ref{notfg}). Observe that
\begin{eqnarray*}
\Var(\Ann(H^i_{\mm}(M))) & = & \Var(\Ann(H^i_{\mm}(M_i/M_{i-1})))\\
                         & =& \Psupp^i(M_i/M_{i-1})\\
                         & =& \Psupp^i(M).
\end{eqnarray*}
Therefore, the result follows by \cite[Theorem 3.1]{TN}.
Corollary \ref{cm1} provides the second step in this sequence. To see the third step, let $\pp \in \Psupp^i(M_i/M_{i-1})$. Thus $H^{i-\dim(R/\pp)}_{\pp R_{\pp}}(M_i/M_{i-1})_{\pp}\neq 0$ and so
\[
\Att\big(H^{i-\dim(R/\pp)}_{\pp R_{\pp}}(M_i/M_{i-1})_{\pp}\big)\neq \emptyset.
\]
By Weak
General Shifted Localization Principle \cite[Exercise 11.3.8]{BS}, there exists
\[
\qq R_{\pp}\in \Att_{R_{\pp}}\big(H_{\pp R_{\pp}}^{i-\dim R/\pp}(M_i/M_{i-1})_{\pp})\big),
\]
where $\qq \subseteq \pp$ and $\qq \in \Att\big(H^i_{\mm}(M_i/M_{i-1})\big)$.  Hence $\qq \in \Att(H^i_{\mm}(M))$. By using Shifted Localization Principle $\qq R_{\pp}\in \Att_{R_{\pp}}\big(H_{\pp R_{\pp}}^{i-\dim R/\pp}(M)_{\pp})\big)$. It follows that $H_{\pp R_{\pp}}^{i-\dim R/\pp}(M)_{\pp})\neq 0$ and hence $\pp \in \Psupp^i M$. The other inclusion is proved in the same way.
\end{proof}
\begin{center}
{Acknowledgment}
\end{center}
\hspace*{\parindent} The author would like to thank the referee
for the useful comments.

\end{document}